\documentclass[12pt, oneside, psamsfonts]{amsart}

\newif\ifPDF
\ifx\pdfoutput\undefined\PDFfalse
\else \ifnum \pdfoutput > 0 \PDFtrue
        \else \PDFfalse
        \fi
\fi

\usepackage[centertags]{amsmath}
\usepackage{amsfonts}
\usepackage{mathrsfs}
\usepackage{textcomp}
\usepackage{amssymb}
\usepackage{amsthm}
\usepackage{newlfont}
\usepackage[all]{xy}

\ifPDF
  \usepackage[pdftex]{color, graphicx}
  \usepackage[pdftex, bookmarks, colorlinks]{hyperref}
  \hypersetup{colorlinks=false}

\else
  \usepackage{color}
  \usepackage[dvips]{graphicx}
  \usepackage[dvips]{hyperref}
\fi

\usepackage{fixltx2e}
\usepackage{tkz-graph}
\tikzset{EdgeStyle/.style = {->}}
\tikzset{LabelStyle/.style= {fill=yellow}}
\usetikzlibrary{shapes,snakes,calendar,matrix,backgrounds,folding}

\usepackage[scale=0.8]{geometry}

\usepackage[pagewise, mathlines, displaymath]{lineno}

\newtheorem{thm}{Theorem}[section]

\newtheorem{lem}[thm]{Lemma}
\newtheorem{prop}[thm]{Proposition}
\theoremstyle{definition}

\theoremstyle{remark}
\newtheorem{rem}[thm]{Remark}

\numberwithin{equation}{section}

\newcommand{\norm}[1]{\left\Vert#1\right\Vert}
\newcommand{\abs}[1]{\left\vert#1\right\vert}

\newcommand{\Real}{\mathbb R}
\newcommand{\Int}{\mathbb Z}
\newcommand{\Comp}{\mathbb C}

\newcommand{\eps}{\varepsilon}

\begin{document}

%OPTIONS FOR LINE NUMBERS----------------------------------
%\linenumbers

\title{A simple separable C*-algebra which is not singly generated}

\author{George A. Elliott}
\address{Department of Mathematics, University of Toronto, Toronto, Ontario, Canada~\ M5S 2E4}
\email{elliott@math.toronto.edu}

\author{Chun Guang Li}
\address{
School of Mathematics and Statistics, Northeast Normal University, Changchun, 130024, P. R.~China
}
\email{licg864@nenu.edu.cn}

\author{Zhuang Niu}
\address{Department of Mathematics, University of Wyoming, Laramie, WY, 82071, USA}
\email{zniu@uwyo.edu}

\thanks{G.A.E.~ is supported by an NSERC Discovery Grant. Z.N.~is supported by a Simons Foundation Grant (MP-TSM-00002606).}

\keywords{Simple C*-algebras, generators, AH algebras}

%\thanks{}
%\keywords{}
\date{\today}
%\dedicatory{}
%\commby{}

%------------------------------------------abstract--------------------------------------

\begin{abstract}
It is shown that there is a simple unital separable AH algebra which is not singly generated.
\end{abstract}

\maketitle

\setcounter{tocdepth}{1}
%\tableofcontents

\section{Introduction}
In \cite{Kadison-list}, among other questions, Kadison asked whether every von Neumann algebra acting on a separable Hilbert space is singly generated (see \cite{Ge_2003} for a survey of Kadison's problem list). Many von Neumann algebras are known to be singly generated, and this generator problem was reduced to the $\mathrm{II}_1$ factors (\cite{Willig_1974}), but it is still open in general. For C*-algebras, the same question has been asked: whether every simple separable C*-algebra is singly generated (without simplicity, C*-algebras such as $\mathrm{C}(S^2)$ provide examples which are not singly generated; see \cite{Nagisa-04}). Stable C*-algebras, UHF-absorbing C*-algebras, $\mathcal Z$-absorbing C*-algebras, as well as Villadsen's algebras of the first type, are known to be singly generated (\cite{Topping_1968}, \cite{Olsen_1976}, \cite{Li-Sh-AD}, \cite{Thiel-Winter-14} \cite{LNR-26}). Clearly, a positive answer for this question of C*-algebras  will imply a positive answer to the von Neumann algebra question. In this note, using Villadsen's second type construction (\cite{Vill-sr}), we show that there are simple separable C*-algebras which are not singly generated (Theorem \ref{generator-thm}). 
\begin{thm}
There are simple unital separable AH algebras $A$ such that $\mathrm{C^*}(1, g) \neq A$ for all $g \in A$. 
\end{thm}

To prove the theorem, let us suppose the AH algebra (constructed in Section \ref{construction}) were generated by a single element. Then, in the base space of a building block of $A$, we consider a distinguished set on which the approximation of the generator has a diagonal block which is cut by the trivial rank-one projection. Then, restricted to this $1 \times 1$ diagonal block, one can show that the projection map from this distinguished set into a given manifold can be approximated by a map factoring through a compact subset of $\Comp$ (the spectrum of the $1 \times 1$ block of the approximating element of the generator). To see that this leads to to a contradiction, we use the Thom-Porteous formula of intersection theory (which was discovered by A.~Toms (\cite{Toms-26}) to have beautiful applications to the study of AH algebras) to show that the projection map above actually is nonzero on the homology group at some degree at least $2$, but such  a homology group of any compact planar set vanishes.

Our example does not have stable rank one or real rank zero (Proposition \ref{srr}). A natural question then is whether every simple (or possibly even not simple in the real rank zero case) separable C*-algebra of stable rank one or real rank zero is singly generated. Note that a positive answer in either cases will still imply a positive answer to Kadison's generator question.

%Note that separable commutative C*-algebras of stable rank one or real rank zero are singly generated.

AI is not used during the research of this work. 

\section{Villadsen's construction of the second type}\label{construction}

%Let us start with some basic facts of complex projective spaces and product spaces. 

Following Villadsen's construction of the second type, let us construct a simple unital AH algebra which will be shown in the next section not to be singly generated. 

\subsection{Complex projective spaces} Recall that $\Comp\mathrm{P}^n$, $n$-dimensional complex projective space, is a compact orientable differentiable manifold of dimension $2n$.  The homology group $\mathrm H_{2i}(\Comp\mathrm{P}^n)$ is $\Int$ for all $0\leq i \leq n$, and all other homology groups are $\{0\}$. Denote by $x_n \in \mathrm H_{2n}(\Comp\mathrm{P}^n)$ the fundamental class of $\Comp\mathrm{P}^n$.  %us choose a generator of $\mathrm H_{2i}(\Comp\mathrm{P}^n)$, and denote it by $x_{n}^{(i)}$. In particular, $x_n^{(n)}$ is the fundamental class of $\Comp\mathrm{P}^n$, and we also denote it by $x_n$.

%The cohomology group $\mathrm H^{2}(\Comp\mathrm{P}^n)$ is isomorphic to $\mathbb Z$. Denote by $\gamma_n$ is a generator of $\mathrm H^{2}(\Comp\mathrm{P}^n)$. Then the cohomology ring $\mathrm H^{*}(\Comp\mathrm{P}^n)$ is isomorphic to $\mathbb Z[\gamma_n]/\left<  \gamma_n^{n+1}\right>$, where 

The cohomology group $\mathrm H^{2i}(\Comp\mathrm{P}^n)$ is $\Int$ for $0 \leq i \leq n$, and all other cohomology groups are $\{0\}$. %Choose $\gamma_n \in \mathrm H^{2}(\Comp\mathrm{P}^n)$ to be the generator such that
Set
$$\gamma_n = c_1(\xi_n) \in  \mathrm H^{2}(\Comp\mathrm{P}^n) ,$$
where $\xi_n$ is the canonical line bundle of $\Comp\mathrm{P}^n$ and $c_1(\cdot)$ is the first Chern class. Then $\gamma_n$ is a generator of $\mathrm H^{2}(\Comp\mathrm{P}^n)$, and its cup power $\gamma_n^i \in \mathrm H^{2i}(\Comp\mathrm{P}^n)$ is a generator of $\mathrm H^{2i}(\Comp\mathrm{P}^n)$. In fact, $\gamma_n$ is the generator of the cohomology ring of $\Comp\mathrm{P}^n$.  By Poincar\'{e} duality, for each $i=0, ..., n$, the element $$x_n^{(n-i)} := \left<\gamma_n^i, x_n \right> \in \mathrm H_{2(n-i)}( \Comp \mathrm{P}^n)$$
is a generator of $\mathrm H_{2(n-i)}( \Comp \mathrm{P}^n)$, where $\left< \cdot, \cdot\right>$ denotes the cap product. (See, for instance, \cite{Hatcher-AT}.)

\subsection{The product space}
Let $M$ be a compact orientable differentiable manifold such that $d = \mathrm{dim}(M) \geq 2$. For each $n = 0, 1, 2, ... $, define 
\begin{equation}\label{base-space}
X_{n} = M \times \underbrace{\Comp\mathrm P^{2} \times \cdots \times \Comp\mathrm P^{2^i} \times \cdots  \times \Comp\mathrm P^{2^{n}}}_{n}.
\end{equation}
Note that
\begin{equation}\label{dim-Xn}
\mathrm{dim}(X_{n}) = d + 2(2 + \cdots + 2^n) = d + 2(2^{n+1} - 2).
\end{equation}
The product space $X_n$, $n=0, 1, 2, ... $, is still a compact orientable manifold, and by the K\"{u}nneth theorem, the fundamental class of $X_n$ is given by
%\begin{eqnarray*}
%\mathrm{H}_{j}(X_{n+1}) & = & \bigoplus_{j_0+2j_1+\cdots +2j_{n} = j} \mathrm{H}_{j_0}(D) \otimes \mathrm{H}_{2j_1}(\mathbb C\mathrm P^2) \otimes \cdots \otimes \mathrm{H}_{2j_n}(\mathbb C\mathrm P^{2^{n}}) \\
%& = & \bigoplus_{j_0+\cdots +j_{n} = j} \Int[c_{j_0} \otimes x_{2}^{(j_1)} \otimes \cdots \otimes  x_{2^{n}}^{(j_{n})}],
%\end{eqnarray*}
%where $c_{j_0} \in \mathrm{H}_{j_0}(D)$,  $x_{2^{k}}^{(j_k)} \in \mathrm H_{2j_k}(\Comp\mathrm{P}^{2^{k}})$, $k=1, ..., n$, are the generator (above).
%In particular, the fundamental class of $X_{n+1}$ is given by 
$$ [X_{n}] :=  e \otimes x_{2} \otimes \cdots \otimes x_{2^i} \otimes \cdots \otimes x_{2^{n}},$$
where %$x_k \in \mathrm{H}_{2k}(\Comp\mathrm P^{k})$ stands for the fundamental class of $\Comp\mathrm P^{k}$ and 
$e \in \mathrm{H}_{d}(M)$ is the fundamental class of $M$.

\subsection{The construction of $A$}
Let $X$ and $Y$ be compact metrizable spaces. Let $\lambda_1, ..., \lambda_n: Y \to X$ be continuous and let $q_1, ..., q_n \in \mathrm{C}(Y)\otimes \mathcal K$ be mutually orthogonal projections, where $\mathcal K$ is the algebra of compact operators on a separable Hilbert space. Then the generalized diagonal map $\mathrm{C}(X) \to \mathrm C(Y) \otimes \mathcal K$ associated with the spectral data $(\lambda_i, q_i)$, $i=1, ..., n$, is defined by
$$ f \mapsto (x \mapsto \sum_{i=1}^n (f \circ \lambda_i)(x)q_i(x)).$$ Let us denote this map just by $(\lambda_i, q_i)_{1\leq i\leq n} = (\lambda_i, q_i)$.

The composition of the two maps associated with spectral data,  $$(\alpha_i, q^{(1)}_i): \mathrm{C}(X) \to \mathrm C(Y) \otimes \mathcal K,\ i=1, ..., n,$$ and 
$$(\beta_j, q^{(2)}_j): \mathrm{C}(Y) \to \mathrm C(Z) \otimes \mathcal K \quad j=1, ..., m, $$ is the map associated with the spectral data $$(\alpha_i\circ\beta_j, \beta_j^*(q^{(1)}_i) \otimes q^{(2)}_j),\quad i=1, ..., n,\ j=1, ..., m,$$ where $p\otimes q$ denotes the tensor product of $p$ and $q$ considered as vector bundles.

Let us construct the C*-algebra $A = \varinjlim A_n$: The base space of $A_n$, $n = 0, 1, ...,$ is given by \eqref{base-space}, i.e., 
%$$X_1 = D, \quad X_{i+1} = X_i \times \mathbb C\mathrm{P}^{d_{i, 1}} \times \cdots \times \mathbb C\mathrm{P}^{d_{i, s_i}}, \quad i=1, 2, ..., $$ %$X_i:=D \times Y_i$,
$$X_0 = M, \quad X_{n} = X_{n-1} \times \mathbb C\mathrm{P}^{2^{n}}, \quad n=1, 2, ....$$ 
Choose $y_n \in X_{n}$ in such a way that for every $n=0, 1, 2, ...$, the sequence of points   
$$y_n, \pi_{n+1}^{(1)}(y_{n+1}), ..., (\pi_{n+1}^{(1)} \circ\cdots\circ \pi_{n+k}^{(1)})(y_{n+k}), ...$$
is dense in $X_n$, where $$\pi_{n+i}^{(1)}: X_{n+i} = X_{n+i - 1} \times \Comp\mathrm P^{2^n} \to X_{n+i-1},\quad i=0, 1, 2, ...,$$ is the projection to the first factor.

Consider the inductive system 
$$
%\xymatrix{
\mathrm{C}(X_0) \to p_1 (\mathrm{C}(X_1) \otimes \mathcal K) p_1 \to \cdots \to  A = \varinjlim A_n, 
%},
$$
where the map $A_{n-1} \to A_{n}$, $n=1, 2, ...$, is induced by the spectral data 
$$(\pi^{(1)}_{n}, \theta)\quad \mathrm{and}\quad (y_{n}, (\pi^{(2)}_{n})^*(\xi_{2^{n}})),$$
where $\theta$ is the (trivial) constant line bundle over $X_n$, and $\pi_n^{(2)}: X_n \to \Comp\mathrm P^{2^n}$ is the standard projection (and recall that $\xi_{2^n}$ is the canonical line bundle over $\Comp \mathrm P^{2^n}$). Then $A$ is a simple unital separable AH algebra.

Recall the notation 
\begin{equation}\label{prod-decomp}
X_{n} = M \times \mathbb C\mathrm{P}^{2} \times \cdots \times \mathbb C\mathrm{P}^{2^{i}} \times \cdots \times \mathbb C\mathrm{P}^{2^{n}},\quad n = 1, 2, ... . 
\end{equation}
The vector bundle (over $X_n$) associated to the projection $p_n$ above is
\begin{equation}\label{decomp-1-bundle}
 \zeta_n: = \theta \oplus \eta_{2} \oplus \cdots \oplus (\underbrace{\eta_{2^{i}}\oplus \cdots \oplus \eta_{2^{i}}}_{2^{i-1}}) \oplus \cdots \oplus (\underbrace{\eta_{2^{n}}\oplus \cdots \oplus \eta_{2^{n}}}_{2^{n-1}}),
 \end{equation}
where $\eta_{2^{i}} = (\pi_{n}^{(i)})^*(\xi_{2^{i}})$, and $\pi_{n}^{(i)}: X_{n} \to \mathbb C\mathrm{P}^{2^{i}}$, $i=1, ..., n$, is the standard projection map associated with  \eqref{prod-decomp}. 

Note that
\begin{equation}\label{chern-class-small}
c_{2^n - 1}(\zeta_n\ominus\theta) = (\pi_{n}^{(1)})^*(\gamma_2) \cup \cdots \cup ((\pi_{n}^{(i)})^*(\gamma_{2^{i}}))^{2^{i-1}} \cup \cdots \cup ((\pi_{n}^{(n)})^*(\gamma_{2^n}))^{2^{n-1}}, 
\end{equation}
where $c_i(\cdot)$ denotes the $i$-th Chern class.

Also note that 
$$ \mathrm{rank}(p_n) = 2^{n}, \quad n=0, 1,  2, ..., $$ 
and the map $\phi_{0, n}: A_0 \to A_n$ is given by
\begin{eqnarray}\label{decp-0-n}
 f & \mapsto & (f \circ\pi_{n}^M) \oplus (f(y_1)\eta_2) \oplus \cdots \oplus (\underbrace{f(y_{i-1})\eta_{2^{i}}\oplus \cdots \oplus f(y_{i-1})\eta_{2^{i}}}_{2^{i-1}}) \\ 
 & & \oplus \cdots  \oplus  (\underbrace{f(y_{n-1})\eta_{2^{n}}\oplus \cdots \oplus f(y_{n-1})\eta_{2^{n}}}_{2^{n-1}}), \nonumber
 \end{eqnarray}
where $\pi_n^M: X_n \to M$ is the standard projection map associated with  \eqref{prod-decomp}.

\section{Generators of $A$}

\subsection{Generators of $A$}
Let us show that the C*-algebra $A$ constructed in the previous section is not singly generated:
\begin{thm}\label{generator-thm}
For every $g \in A$, one has $\mathrm{C}^*(1, g) \neq A$. In fact, there is a finite set $\mathcal E \subseteq A$ such that for every $g \in A$, the sub-C*-algebra $\mathrm{C^*}(1, g)$ does not contain $\mathcal E$.
\end{thm}

\begin{proof}
%Choose $p^{(i)}_j$ to be line bundles. 
First, regard $$M \subseteq \Comp^{\tilde{d}}$$ as a submanifold for some $\tilde{d} \in \mathbb N$ (where $\Comp$ is only regarded as the manifold $\Real^2$), and consider the coordinate maps $$e_i: M \to \Comp,\  i=1, ..., \tilde{d}, $$
which we shall regard as elements of $A_0 = \mathrm{C}(M)$.

Let us show that for every $g \in A$, the sub-C*-algebra $\mathrm{C^*}(g)$ does not contain $$\mathcal E  := \{e_1, ..., e_{\tilde{d}}\}. $$ In particular, the C*-algebra $A$ is not singly generated.

First, note that there is $\eps>0$ such that
\begin{enumerate}
\item there is a retraction from the $\eps \tilde{d}$-neighbourhood of $M \subseteq \Comp^{\tilde{d}}$ to $M$, and 
\item if $f, h: Z \to M$ are continuous maps such that $\norm{f(x) - h(x)} < \eps$ for all $x \in Z$, where $Z$ is a compact space, then $[f]_* = [h]_*$ on $\mathrm{H}_*(Z)$.
\end{enumerate}

Now, assume that the set $\mathcal E$ were contained in a singly generated sub-C*-algebra. Then, there would be (noncommutative) polynomials $P_1, ..., P_{\tilde{d}}$ such that for sufficiently large $n$, there is $g \in A_{n} = p_n (\mathrm{C}(X_{n}) \otimes \mathcal K) p_n $ such that
\begin{equation}\label{almost-gen}
 \phi_{0, n}(e_i) \approx_\eps P_i(g),\quad i=1, 2, ..., \tilde{d}.
\end{equation} 

Write
$$ g =(q_1 + \cdots + q_{2^n}) g (q_1 + \cdots + q_{2^{n}}) =  \sum_{i, j = 1}^{2^n} q_i g q_j,$$
where $q_1, ..., q_{2^n}$ are the partition of unity in $A_n$ given by the rank-one projections of the line bundles of \eqref{decomp-1-bundle}. %Moreover, the element $g$ can be chosen to be sufficiently generic in the sense that the bundle maps $q_igq_j$ are regular for all $i, j = 1, ..., 2^n$. 

Regard $$(q_1gq_i)^* = q_i g^*q_1\quad \mathrm{and} \quad q_i gq_1, \quad i=2, ..., 2^n, $$ as morphisms from the trivial line bundle $\theta$ to the line bundle corresponding to $q_i$, %$\eta_\bullet$, 
and then consider the map 
$$\sigma:=(q_2g^*q_1, ..., q_{2^n}g^*q_1, q_2gq_1, ..., q_{2^n}gq_1)$$
which is a morphism from the trivial line bundle $$E: = \theta$$ to the vector bundle 
$$F:= %(q_2\oplus \cdots \oplus q_{2^n}) \oplus (q_2\oplus \cdots \oplus q_{2^n})  = 
(\zeta_n \ominus\theta) \oplus (\zeta_n \ominus\theta). $$ 
(Recall $\zeta_n$ is given by \eqref{decomp-1-bundle}.)
Note that 
$$\mathrm{rank}(F) = 2\cdot \mathrm{rank}(\zeta_i\ominus\theta) = 2 (2^n-1), $$
and
by \eqref{chern-class-small},
\begin{equation}\label{chern-class-double}
c_{2(2^n - 1)}(F) = (c_{2^n - 1}(\zeta_n\ominus\theta))^2 = (\pi_{n}^{(1)})^*(\gamma_2^2) \cup \cdots \cup (\pi_{n}^{(i)})^*(\gamma^{2^{i}}_{2^{i}}) \cup \cdots \cup (\pi_{n}^{(n)})^*(\gamma^{2^{n}}_{2^n}) \neq 0. 
\end{equation}

Since $E$ is the trivial line bundle, the morphism $\sigma$ also can be identified as a continuous section of $F$: $$\sigma: X_{n} \ni x \mapsto \sigma_x(1_\Comp) \in F_x\subseteq F.$$ Since $ c_{2(2^n - 1)}(F) \neq 0$, the section $\sigma$ has non-empty intersection with the zero section $0_F \subseteq F$. By the Whitney approximation theorem and the transversality theorem (\cite{Hirsch-GTM}), the section $\sigma$ can be perturbed within arbitrarily small tolerance to a smooth section such that $\sigma$ is transverse to the zero section $0_F \subseteq F$. 
%\vskip 1in the map $\sigma$ must have zero point, and then, after an arbitrarily small perturbation, one may assume that $\sigma$ is differentiable and the image of $\sigma$ contains a small open neighbourhood of $0$. By Sard's theorem, after another perturbation, one then may assume that $0$ is a regular value of $\sigma$. 
Thus, one may assume that the set $$Z(g) :=   \{x \in X_{n}: \sigma(x) = 0 \}=\sigma^{-1}(0_F)$$ is a differentiable (orientable) submanifold of $X_{n}$ with codimension %equal to the codimesion of $0_F$ inside $F$
 $2 \cdot \mathrm{rank}(F) = 2(2^{n+1} - 2)$.

%Then let us assume that $g$ is generic in the sense that of the zero locus of $\sigma$
%$$Z(g) := D_0(\sigma) :=  \{x \in X_n: \mathrm{rank}(\sigma(x)) = 0 \}$$
%is a differential  submanifold of $X_{n+1}$ with codimension $2(2(2^n - 1))$.

\subsubsection*{Over the set $Z(g)$}

Rewrite the set $Z(g)$ as
\begin{equation}\label{defn-Z}
Z(g) = \{x \in X_{n}: q_1gq_i(x) = 0,\ q_igq_1(x) = 0,\ i=2, ..., 2^n\}.
\end{equation}
In other words, over the set $Z(g)$, the element $g$ has the diagonal form $q_1gq_1 + q_1^\perp g q_1^\perp$. %Since $g$ is chosen to be regular (generic), the set $Z(g)$ is a differential submanifold of $X_{n+1} = D \times Y_{n}$ with codimension $(2(2^n - 1))$.  %with co-dimension $2(2^n-1)$.

For each function $e_i$, $i=1, ..., \tilde{d}$, note that, by \eqref{almost-gen}, 
\begin{eqnarray*}
q_1\phi_{0, n}(e_i) q_1 & \approx_\eps & q_1 P_i((q_1 + q_1^\perp) g (q_1 + q_1^\perp)) q_1 \\
& = & P_i(q_1 g q_1) + q_1(\sum q_1 \cdots q_1^\perp \cdots q_1) q_1.
\end{eqnarray*}
Restricting to $Z(g)$, one has
\begin{equation}\label{pre-approx-0}
(q_1 \phi_{0, n}(e_i) q_1)|_{Z(g)} \approx_\eps P_i((q_1gq_1)|_{Z(g)}), \quad i=1, ..., \tilde{d}.
\end{equation}
Note that, by \eqref{decp-0-n}, 
$$q_1 \phi_{0, n}(f) q_1 = f \circ \pi_{n}^M,\quad f \in \mathrm{C}(M),$$
and hence
\begin{equation}\label{pre-approx-1}
(e_i \circ \pi_{n}^M)(x) \approx_\eps P_i((q_1gq_1)(x)), \quad x \in Z(g),\  i=1, ..., \tilde{d},
\end{equation}
where, since $q_1$ is the (trivial) constant rank-one projection, the $q_1gq_1$ is regarded naturally as a continuous function on $X_{n}$.  

Consider the maps
$$ \Pi:= (e_1\circ\pi_{n}^{M}, e_2 \circ \pi_{n}^M, ... , e_{\tilde{d}} \circ \pi_{n}^M): Z(g) \to  M \subseteq \Comp^{\tilde{d}}$$
and
$$ \Psi:= (P_1(q_1gq_1), P_2(q_1gq_1), ..., P_{\tilde{d}}(q_1gq_1)): Z(g) \to \mathrm{sp}(q_1gq_1) \to  \Comp^{\tilde{d}}. $$
By \eqref{pre-approx-1},
\begin{equation}\label{main-approx}
\Pi(x) \approx_{\eps \tilde{d}} \Psi(x),\quad x \in Z(g).
\end{equation}
%$$(q_1e_1q_1, q_1e_2q_1, ... , q_1e_{\tilde{d}}q_1) \approx_{\eps \tilde{d}} (P_1(q_1gq_1), P_2(q_1gq_1), ..., P_{\tilde{d}}(q_1gq_1)). $$

Note that the map $\Pi$ is just the projection to $M$,
\begin{equation}\label{large-zero-set}
\xymatrix{
Z(g) \ar[r]^-\iota & X_{n} = M \times Y_n \ar[r]^-{\pi_{n}^M} &  M \subseteq \Comp^{\tilde{d}},
}
\end{equation}
%and note that the induced map
%$$\mathrm{H}^d(D) \to \mathrm{H}^d(Z_n)$$
%is non-zero.
where $Y_n := \mathbb C\mathrm{P}^{2} \times \cdots \times \mathbb C\mathrm{P}^{2^{i}} \times \cdots \times \mathbb C\mathrm{P}^{2^{n}}$,
and also note that the map $\Psi$ factors through the planar set $\mathrm{sp}(q_1gq_1) \subseteq \Comp$.

By the choice of $\eps$, we may assume that the image of $\Psi$ is actually inside $M$, and then, by \eqref{main-approx} and the choice of $\eps$ again, one has
$$[\Pi]_* = [\Psi]_*: \mathrm{H}_*(Z(g)) \to \mathrm{H}_*(M).$$
%the maps $\Pi$ and $\Psi$ are homotopic (as maps from $Z(g)$ to $D$).

%it is homotopic to a map $Z \to D$ which factors through the compact set $\mathrm{sp}(q_1gq_1) \subseteq \Comp$ (since $q_1$ is the trivial rank one projection, the element $q_1gq_1$ belongs to $\mathrm{C}(X_n)$; in particular, it is normal.). %Then the first map is homotopic to a map which factors through the compact set $\mathrm{sp}(q_1gq_1) \subseteq \Comp$. 

%and thus the induced map
%$$\mathrm{H}^d(D) \to \mathrm{H}^d(X_n)$$
%is zero ($d \geq 2$), which is a contradiction. 

%Let us now see why the map \eqref{large-zero-set} is not zero on $\mathrm{H}^d(\cdot)$.

\subsubsection*{The fundamental class of $Z(g)$ and the Thom-Porteous formula}
Recall that, by the choice of $g$, the set $Z(g)$ is a submanifold of $X_{n}$ with codimension $2(2^{n+1} - 2)$, and the image of the fundamental class of $[Z(g)]$ in $\mathrm{H}_{m}(X_{n+1})$,  %$$\mathbb{D}_0 \in \mathrm{H}_{m}(D_0(\sigma)),$$
where $m = \mathrm{dim}(X_{n}) - 2(2^{n+1}-2) = d =  \mathrm{dim}(M),$  
is independent of $\sigma$, and is given by the (special case of the) Thom-Porteous formula,
\begin{equation}\label{TP-identity}
\iota_*([Z(g)]) = \left< c_{2(2^n-1)}(F - E), [X_{n}]\right>,
\end{equation} 
where
$$ c(F - E) = c(F)/c(E).$$

%Let us choose $p_i^j$ carefully such that \eqref{TP-identity} does not hold, and therefore $A$ cannot be singly generated.

%\section{Computations}

Let us calculate the right hand side of \eqref{TP-identity}: By \eqref{chern-class-double} and note that $E$ is trivial, we have 
\begin{eqnarray}\label{comp-right}
&& \left< c_{2(2^n-1)}(F - E)), [X_{n}] \right> \nonumber  \\
& = & \left<  ((\pi_{n}^{(1)})^*(\gamma_2))^2 \cup \cdots \cup ((\pi_{n}^{(i)})^*(\gamma_{2^{i}}))^{2^{i}} \cup \cdots \cup ((\pi_{n}^{(n)})^*(\gamma_{2^{n}}))^{2^{n}}, e \otimes x_{2} \otimes \cdots \otimes x_{2^{n}} \right> \nonumber \\
& = &e \otimes \left<\gamma^2_2, x_2 \right> \otimes \cdots \otimes \left< \gamma_{2^{i}}^{2^{i}}, x_{2^{i}}\right> \otimes \cdots \otimes \left< \gamma_{2^{n}}^{2^{n}}, x_{2^{n}}\right> \nonumber \\
& = & e \otimes 1 \otimes \cdots \otimes 1 \in \mathrm{H}_{d}(X_n) \setminus\{0\}. 
\nonumber
\end{eqnarray}
This implies that % the composition
%$$
%\nu: 
%\xymatrix{
%Z(g) \ar[r]^-\iota & D \times Y_n \ar[r]^-\pi & D
%}.
%$$
%Then 
$$(\Pi)_d([Z(g)]) = (\pi_{n}^M)_d(\iota([Z(g)])) = (\pi_{n}^M)_d(e \otimes 1\otimes \cdots \otimes 1) = e \neq 0.$$
But the map $\Pi$ is homotopic to the map $\Psi$ which factors through a planar set, and thus $$(\Pi)_{d}([Z(g)]) = (\Psi(Z(g)))_{d} = 0, $$ which is a contradiction.

Thus, the set $\mathcal E$ is not contained in any singly generated sub-C*-algebra, as desired.
\end{proof}

%In the rest of this section, let us show that the stable rank of $A$ is strictly larger than $1$.

\subsection{Stable rank and real rank of $A$}

Let us show that both the stable rank of $A$ and the real rank of $A$ are at least $2$.

Let $f: M \to N$ be a continuous map, where $M$ and $N$ are orientable manifolds with the same dimension $d$.  Fix orientations of $M$ and $N$. Let $y \in N$ be a regular point. Then the degree of $f$ at $y$, denoted by $\mathrm{deg}(f, y)$, is defined by $$\mathrm{deg}(f, y) = \sum_{x \in f^{-1}(y)} w_x,$$ where $w_x = \pm 1$ depending on whether $\mathrm{d}f(x): \mathrm{T}_M(x) \to \mathrm{T}_N(y)$ preserves the orientation or reverses it. If $N$ is connected, then $\mathrm{deg}(f, y)$ is independent of the choice of $y$, and it is denoted by $\mathrm{deg}(f)$, and it satisfies $f_*(\mu) = \mathrm{deg}(f) \nu$ (note that $\mathrm{H}_d(N) \cong \Int$ in this case). (See, for instance, \cite{Hirsch-GTM}.)

%fundamental classes $\mu \in \mathrm{H}_d(M)$ and $\nu \in \mathrm{H}_d(N) \cong \Int$. The degree of $f$ equals to the integer $\mathrm{deg}(f)$ such that $f_*(\mu) = \mathrm{deg}(f) \nu$. 

\begin{lem}\label{nenzero-dc-lem}
Let $M$ and $N$ be differentiable manifolds with the same dimension, and assume that $N$ is connected. Let $L \subseteq N$ be a connected differentiable submanifold, and let $f: M \to N$ be a differentiable map which transverses to $L$. Assume that $\mathrm{deg}(\pi) = 1$. Then $\mathrm{deg}(f |_{f^{-1}(L)}) \neq 0$.
\end{lem}

\begin{proof}
Since $f$ transverses to $L$, the preimage $f^{-1}(L)$ is a submanifold of $M$ with same dimension as $L$. Fix an orientation of $f^{-1}(L)$, and then fix an orientation of $M$ which is compatible with $f^{-1}(L)$. Also fix an orientations of $L$ and $N$ which are compatible.

Pick a regular value $y \in L \subseteq N$ for $f: f^{-1}(L) \to L$. Then $f^{-1}(y)$ is a zero-dimensional submanifold of $M$, and hence $$f^{-1}(y) = \{x_1, x_2, ..., x_n\}.$$ 

Consider the tangent spaces $$\mathrm{T}_M(x_i), \quad i=1, ..., n, $$ and the tangent maps
$$\mathrm{d}f(x_i): \mathrm{T}_M(x_i) \to \mathrm{T}_N(y). $$

Then
$$
\mathrm{deg}(f) = \sum_{i=1}^n w_i,
$$
where $w_i = \pm 1$, $i=1, ..., n$, depending on whether the tangent map $\mathrm df: \mathrm{T}_M(x_i) \to \mathrm{T}_N(y)$ preserves the orientation or reverses it. Since $\mathrm{deg}(f) = 1$, necessarily $n$ is odd.

On the other hand
$$
\mathrm{deg}(f|_{f^{-1}(L)}) = \sum_{i=1}^n w'_i,
$$
where $w'_i = \pm 1$, $i=1, ..., n$, depends on whether the tangent map $\mathrm d(f|_{f^{-1}(L)}): \mathrm{T}_{f^{-1}(L)}(x_i) \to \mathrm{T}_L(y)$ preserves the orientation or reverses the orientation. Since $n$ is odd, the sum $\sum_{i=1}^n w'_i$ must be nonzero, as desired.
%
% \mathrm{T}_M(x_2), ..., \mathrm{T}_M(x_n). $$
\end{proof}

\begin{prop}\label{srr}
With $A$ the C*-algebra considered above, the following bounds on the stable rank and real rank hold:  $$2 \leq \mathrm{tsr}(A) \leq \left\{ 
\begin{array}{ll}
3, & d \leq 4, \\
4, & d > 4;
\end{array}
\right.$$
and 
$$  2 \leq \mathrm{rr}(A) \leq 
\left\{ 
\begin{array}{ll}
2, & d = 2, \\
3, & d > 2.
\end{array}
\right.
$$
(Recall that $d = \mathrm{dim}(M) \geq 2$.)
\end{prop}

\begin{proof}
\noindent {\em Stable Rank.} Let us work on the stable rank first.

By \cite{Nistor-tsr}, $$\mathrm{tsr}(A_n) = \lceil \lfloor  2^{n+1} -2 + d/2 \rfloor /2^n \rceil + 1,$$ and therefore
$$ \mathrm{tsr}(A) \leq \left\{ 
\begin{array}{ll}
3, & d \leq 4, \\
4, & d > 4.
\end{array}
\right.$$

Let us show that $ \mathrm{tsr}(A) \geq 2$. It is enough to find an element of $A$ which is not in the closure of the invertible elements. 

On the manifold $M$, choose a two-dimensional orientable connected submanifold $S$ (if $\mathrm{dim}(M) = 2$, just choose $S$ to be a connected component of $M$; otherwise, $S$ may be chosen to be a  two-dimensional sphere). On $S$, choose $S^+ \subseteq S$ which is homeomorphic to the unit disk, and fix a homeomorphism $$a: S^+ \to \{z \in \Comp: \abs{z}\leq 1\}.$$ Extend $a$ to a continuous function from $M$ to the unit disk, and still denote it by $a$. Regard $a$ as an element of $\mathrm{C}(M) = A_1$.

Suppose there were $g \in A_n$ for a sufficiently large $n$ that $g$ is invertible and 
\begin{equation}\label{approx-inv-0}
\norm{\phi_{0, n}(a) - g} < 1.
\end{equation} 
Write $$A_n = (q_1+\cdots + q_{2^n}) (\mathrm{C}(X_n)\otimes \mathcal K) (q_1+\cdots + q_{2^n}),$$ and consider the set $Z(g) \subseteq X_n=M \times Y_n$ as below:
$$Z(g) = \{x \in X_n: q_1gq_i(x) = q_1g^*q_i(x) = 0,\ i=2, ..., 2^n\}.$$ With the same argument as in the proof of Theorem {generator-thm}, by the transversality theorem, the element $g$ can be chosen further that $Z(g) \subseteq X_n$ is a differentiable (orientable) submanifold of $X_n$ with the same dimension as $M$. Note that, restricted to $Z(g)$, the element $g$ has the diagonal form $q_1 g q_1 + q_1^\perp g q_1^\perp$.
Using the Thom-Porteous formula, the same computation as in the proof of  Theorem \ref{generator-thm} shows that
$$(\pi_n^M)_*[Z(g)] = [M],$$
and therefore, restricted to the connected component of $M$ which contains $S$, one has 
$$\mathrm{deg}(\pi_n^M|_{Z(g)}) = 1.$$

%Consider the set $Z(g) \subseteq M \times Y_n$ similarly as above, and with suitable choice of $g$, we may assume that $Z(g) \subseteq M \times Y_n$ is a differentiable submanifold with dimension $\mathrm{dim}(M)$. 

Let us consider the cut-down $q_1gq_1$. Since $q_1$ is the rank-one trivial projection, it is a function on $X_n$. To simply the notation, let us still denote it by $g$. Then, restricted to $Z(g)$, by \eqref{approx-inv-0}, one has 
%Then, restricted to $Z(g)$, there is a smooth function, still denoted by $g$, such that
$$\norm{ a \circ \pi^M_{n}|_{Z(g)} - g} < 1.$$ 
By the transversality theorem, there is a map $\tilde{\pi}: Z(g) \to M$ which is arbitrarily close to $\pi^M_{n}$ and transverse to $S \subseteq M$. Then, $\tilde{\pi}$ can be chosen such that 
\begin{equation}\label{approx-inv}
  \norm{ a \circ \tilde{\pi} - g} < 1, 
\end{equation}  
$$\mathrm{deg}(\tilde{\pi}) = \mathrm{deg}(\pi^M_{n}|_{Z(g)}) = 1, $$
and the set 
$$ \Sigma: = \tilde{\pi}^{-1}(S)$$ is a 2-dimensional orientable submanifold of $Z(g)$.

Define
$$\Sigma^1:= \tilde{\pi}^{-1}(S^1),\quad \Sigma^+:=\tilde{\pi}^{-1}(S^+), \quad \Sigma^-:=\tilde{\pi}^{-1}(S^-),$$
where
$$S^-=S \setminus \mathrm{int}(S^+) \quad \mathrm{and} \quad S^1 = S^+ \cap S^-. $$
Then, applying the Mayer-Vietoris sequence (and its naturality), one has the maps between exact sequences 
$$
\xymatrix{
\cdots \ar[r] & \mathrm{H}_2(\Sigma) \ar[r]^{\partial_\Sigma} \ar[d]^{(\tilde{\pi}|_{\Sigma})_2} & \mathrm{H}_1(\Sigma^1) \ar[r]^-{(\iota^+, \iota^-)} \ar[d]^{(\tilde{\pi}|_{\Sigma^1})_1} & \mathrm{H}_1(\Sigma^+) \oplus  \mathrm{H}_1(\Sigma^-) \ar[r] \ar[d] & \cdots \\
\cdots \ar[r]  & \mathrm{H}_2(S) \ar[r]^{\partial_{S}} & \mathrm{H}_1(S^1) \ar[r]^-{(\iota^+, \iota^-)} & \mathrm{H}_1(S^+) \oplus  \mathrm{H}_1(S^-) \ar[r]  & \cdots
},
$$
where $\iota^{\pm}$ are the embeddings of the corresponding spaces.

By Lemma \ref{nenzero-dc-lem}, $$\mathrm{deg}(\tilde{\pi}|_{\Sigma}) \neq 0,$$ and hence the map $(\tilde{\pi}|_{\Sigma})_2$ is nonzero.
Since the boundary map $\partial_{S}$ is an isomorphism, the element $$\sigma := \partial_\Sigma(\mu) \in \mathrm{H}_1(\Sigma^1)$$
is not zero, where $\mu \in \mathrm{H}_2(\Sigma)$ is the fundamental class.  
Note that, by the exactness, 
\begin{equation}\label{zero-2-dim}
\iota^+(\sigma) = 0 \in \mathrm{H}_1(\Sigma^+).
\end{equation}

Note that the restriction of $a$ to $S^+$ is a homeomorphism to the disc which sends $S^1$ to the unit circle in $\Comp$, and also note that $$ \mathrm{H}_1(\Comp\setminus\{0\}) \ni (a\circ \tilde{\pi}|_{\Sigma^1})_1(\sigma) \neq 0. $$
Consider the map $$g: Z(g) \to \Comp\setminus\{0\}.$$ By \eqref{approx-inv}, its restriction to $\Sigma^1$ is homotopic to $(a\circ \tilde{\pi})|_{\Sigma^1}$ (as a map to $\Comp\setminus\{0\}$), and therefore  
$$(g|_{\Sigma^1})_1(\sigma) = (a\circ \tilde{\pi}|_{\Sigma^1})_1(\sigma) \neq 0.$$

On the other hand, the map $g|_{\Sigma^1}$ factors through
$$\xymatrix{
\Sigma^1 \ar[r]^-{\iota^+} & \Sigma^+ \ar[r]^-{g} & \Comp\setminus\{0\},
}
$$
and hence, by \eqref{zero-2-dim},
$$ (g|_{\Sigma^1})_1(\sigma) = (g)_1(\iota^+(\sigma)) = 0,$$
which is a contradiction. Therefore, the element $a$ cannot be approximated by invertible elements within distance $1$. 

\noindent {\em Real Rank.} Let us now work on the real rank of $A$. The argument is almost the same as the argument for stable rank.

By \cite{BE-rr} (and argument of \cite{Nistor-tsr}; see proof of Theorem 10 of \cite{Vill-sr}), 
$$\mathrm{rr}(A_n) = \lceil \frac{ d+ 2^{n+2} - 4}{2^{n+1} - 1} \rceil =  \lceil 2+  \frac{ d - 2}{2^{n+1} - 1} \rceil, $$
and therefore
$$ \mathrm{rr}(A) \leq \left\{ 
\begin{array}{ll}
2, & d = 2, \\
3, & d > 2.
\end{array}
\right.$$

Let us show that $\mathrm{rr}(A) \geq 2$. As the case of stable rank, choose a two-dimensional (orientable) submanifold $S$, and choose $S^+ \subseteq S$ which is homeomorphic to the unit disk. Fix a homeomorphism $a$ from $S^+$ to the unit disk, and extend it to $M$. Then consider the real-valued functions 
$$ \Re(a), \ \Im(a) \in \mathrm{C}(M) = A_1. $$  

Let us show that there do not exist self-adjoint elements $b_1, b_2 \in A$ such that 
\begin{equation}\label{perturb-in-rr-0}
 \norm{(b_1^2 + b_2^2) - (\phi_{0, \infty}(\Re(a))^2 + \phi_{0, \infty}(\Im(a))^2)} < 1
\end{equation} 
and $b_1^2 + b_2^2$ is invertible. 

Assume such self-adjoint elements $b_1$ and $b_2$ exist. Without loss of generality, one may assume that $$b_1, b_2 \in A_n = (q_1+\cdots + q_{2^n}) (\mathrm{C}(X_n)\otimes \mathcal K) (q_1+\cdots + q_{2^n})$$ for some $n$, and then, consider the set $Z(b_1, b_2) \subseteq X_n=M \times Y_n$ as below:
$$Z(b_1, b_1) = \{x \in X_n: q_1b_1q_i(x) = q_1b_2q_i(x) = 0,\ i=2, ..., 2^n\},$$ and one may assume further that $Z(b_1, b_2) \subseteq X_n$ is a differentiable (orientable) submanifold of $X_{n}$ with the same dimension as $M$. Note that, restricted to $Z(b_1, b_2)$, both $b_1$ and $b_2$ have the diagonal form $q_1 \cdot q_1 + q_1^\perp \cdot q_1^\perp$.
Using the Thom-Porteous formula, the same computation as Theorem \ref{generator-thm} shows that
$$\mathrm{deg}(\pi_n^M|_{Z(b_1, b_2)}) = 1.$$

Let us consider the cut-downs $q_1b_1q_1$ and $q_1b_2q_1$. Since $q_1$ is the rank-one trivial projection, both are functions on $X_n$. To simply the notation, let us still denote them by $b_1$ and $b_2$. Then, restricted to $Z(b_1, b_2)$, by \eqref{perturb-in-rr-0}, one has 
\begin{equation}
\norm{(b_1^2 + b_2^2) - ((\Re(a) \circ \pi_n^M|_{Z(b_1, b_2)})^2 + (\Im(a) \circ \pi_n^M|_{Z(b_1, b_2)})^2) } < 1
\end{equation}

%$$ \norm{\Re(a) \circ \pi_n^M|_{Z(b_1, b_2)} - b_1} < 1 \quad \mathrm{and} \quad  \norm{ \Im(a) \circ \pi_n^M|_{Z(b_1, b_2)} - b_2} < 1. $$

By the transversality theorem, there is a map $\tilde{\pi}: Z(b_1, b_2) \to M$ which is arbitrarily close to $\pi^M_{n}$ and transverses to $S \subseteq M$. Then, $\tilde{\pi}$ can be chosen such that 
\begin{equation}\label{approx-inv-rr}
\norm{(b_1^2 + b_2^2) - ((\Re(a) \circ \tilde{\pi})^2 + (\Im(a) \circ \tilde{\pi})^2) } < 1
\end{equation}
%$$(\pi^M_{n}|_{Z(b_1, b_2)})_* = (\tilde{\pi})_*, $$
$$\mathrm{deg}(\tilde{\pi}) = \mathrm{deg}(\pi^M_{n}|_{Z(b_1, b_2)}) = 1, $$
and the set 
$$ \Sigma: = \tilde{\pi}^{-1}(S)$$ is a 2-dimensional submanifold of $Z(b_1, b_2)$. Then the same argument as in the stable rank case applyed to the function $$g:=b_1 + i b_2:  Z(b_1, b_2) \to \Comp$$
leads to a contradiction with the vanishing of the homology class $g|_{\Sigma_1}(\sigma)$. This shows that such $b_1$ and $b_2$ do not exist. Therefore the real rank of $A$ is at least $2$.
\end{proof}

\begin{rem}
%The same argument as \cite{Vill-sr} shows that
%$$2 \leq \mathrm{tsr}(A) \leq (3, 4). $$
Is every simple separable C*-algebra of stable rank one singly generated? (Without simplicity, the C*-algebra $\mathrm{C}(X)$, where $X$ is a one-dimensional non-planar compact metrizable space, provides an example of a separable C*-algebra of stable rank one but is not singly generated; see Proposition 2 of \cite{Nagisa-04}.) Is every separable C*-algebra (not necessarily simple) of real rank zero singly generated? 
\end{rem}

\bibliographystyle{plainurl}
%\bibliography{operator_algebras}

\begin{thebibliography}{10}

\bibitem{BE-rr}
E.~J.~Beggs and D.~E.~Evans.
\newblock The real rank of algebras of matrix valued functions.
\newblock {\em Internat. J. Math.}, 2(2):131--138, 1991.
\newblock \href {https://doi.org/10.1142/S0129167X91000089}
  {\path{doi:10.1142/S0129167X91000089}}.

\bibitem{Ge_2003}
L.~M.~Ge.
\newblock On ``{Problems on von Neumann Algebras by R. Kadison, 1967}''.
\newblock {\em Acta Mathematica Sinica, English Series}, 19(3):619--624, 2003.
\newblock URL: \url{http://dx.doi.org/10.1007/s10114-003-0279-x}, \href
  {https://doi.org/10.1007/s10114-003-0279-x}
  {\path{doi:10.1007/s10114-003-0279-x}}.

\bibitem{Hatcher-AT}
A.~Hatcher.
\newblock {\em Algebraic topology}.
\newblock Cambridge University Press, Cambridge, 2002.

\bibitem{Hirsch-GTM}
M.~W.~Hirsch.
\newblock {\em Differential Topology}.
\newblock Number~33 in Graduate Texts in Mathematics. Springer-Verlag, 1976.

\bibitem{Kadison-list}
R.~V.~Kadison.
\newblock Problems on von {Neumann} algebras.
\newblock unpublished manuscript, presentedat Conference on Operator Algebras
  and Their Applications, Louisiana State Univ., Baton Rouge, La., 1967.

\bibitem{LNR-26}
C.~G.~Li, Z.~Niu, and V.~Ruzicka.
\newblock Villadsen algebras are singly generated.
\newblock {\em J. Noncommut. Geom., accepted.}, 2026.

\bibitem{Li-Sh-AD}
W.~H.~Li and J.~H.~Shen.
\newblock A note on approximately divisible {C*}-algebras.
\newblock 04 2008.
\newblock URL: \url{https://arxiv.org/pdf/0804.0465.pdf}, \href
  {https://arxiv.org/abs/0804.0465} {\path{arXiv:0804.0465}}.

\bibitem{Nagisa-04}
M.~Nagisa.
\newblock Single generation and rank of {$C^*$}-algebras.
\newblock In {\em Operator algebras and applications}, volume~38 of {\em Adv.
  Stud. Pure Math.}, pages 135--143. Math. Soc. Japan, Tokyo, 2004.
\newblock \href {https://doi.org/10.2969/aspm/03810135}
  {\path{doi:10.2969/aspm/03810135}}.

\bibitem{Nistor-tsr}
V.~Nistor.
\newblock Stable rank for a certain class of type {I} {C*}-algebras.
\newblock {\em J. Operator Theory}, 17(2):365--373, 1987.

\bibitem{Olsen_1976}
C.~L.~Olsen and W.~R.~Zame.
\newblock Some {C*}-algebras with a single generator.
\newblock {\em Transactions of the American Mathematical Society}, 215:205,
  January 1976.
\newblock URL: \url{http://dx.doi.org/10.2307/1999722}, \href
  {https://doi.org/10.2307/1999722} {\path{doi:10.2307/1999722}}.

\bibitem{Thiel-Winter-14}
H.~Thiel and W.~Winter.
\newblock The generator problem for {$\mathcal Z$}-stable {C*}-algebras.
\newblock {\em Trans. Amer. Math. Soc.}, 366(5):2327--2343, 2014.
\newblock \href {https://doi.org/10.1090/S0002-9947-2014-06013-3}
  {\path{doi:10.1090/S0002-9947-2014-06013-3}}.

\bibitem{Toms-26}
A~S.~Toms.
\newblock Schubert calculus and uniform property {$\Gamma$}.
\newblock 06 2026.
\newblock URL: \url{https://arxiv.org/pdf/2606.12188.pdf}, \href
  {https://arxiv.org/abs/2606.12188} {\path{arXiv:2606.12188}}.

\bibitem{Topping_1968}
D.~M.~Topping.
\newblock {UHF} algebras are singly generated.
\newblock {\em Math. Scand.}, 22:224, 1968.
\newblock URL: \url{http://dx.doi.org/10.7146/math.scand.a-10886}, \href
  {https://doi.org/10.7146/math.scand.a-10886}
  {\path{doi:10.7146/math.scand.a-10886}}.

\bibitem{Vill-sr}
J.~Villadsen.
\newblock On the stable rank of simple {C*}-algebras.
\newblock {\em J. Amer. Math. Soc.}, 12(4):1091--1102, 1999.
\newblock URL: \url{http://dx.doi.org/10.1090/S0894-0347-99-00314-8}, \href
  {https://doi.org/10.1090/S0894-0347-99-00314-8}
  {\path{doi:10.1090/S0894-0347-99-00314-8}}.

\bibitem{Willig_1974}
P.~Willig.
\newblock Generators and direct integral decompositions of {W*}-algebras.
\newblock {\em Tohoku Mathematical Journal}, 26(1), January 1974.
\newblock URL: \url{http://dx.doi.org/10.2748/tmj/1178241231}, \href
  {https://doi.org/10.2748/tmj/1178241231} {\path{doi:10.2748/tmj/1178241231}}.

\end{thebibliography}

\end{document}